\documentclass{amsart}
\usepackage{enumitem}
\usepackage{amsrefs}
\newtheorem{theorem}{Theorem}[section]
\newtheorem{lemma}[theorem]{Lemma}
\newtheorem{corollary}[theorem]{Corollary}

\numberwithin{equation}{section}

\begin{document}
\title[Fourier transform in weighted r.i. spaces]{The Fourier transform in weighted rearrangement invariant spaces}
\author{Mieczys{\l}aw Masty{\l}o}
\address{Faculty of Mathematics and Computer Science,
Adam Mickiewicz University;
Uniwersytetu Pozna\'nskiego 4, 61-614 Pozna\'n, Poland}
\email{mieczyslaw.mastylo@amu.edu.pl}
\thanks{The first author was supported by the National Science Centre, Poland, Project no. 2019/33/B/ST1/00165}

\author{Gord Sinnamon}
\address{Department of Mathematics,
Western University,
London, Canada}
\email{sinnamon@uwo.ca}
\thanks{The second author was supported by the Natural Sciences and Engineering Research Council of Canada}


\begin{abstract} It is shown that if the Fourier transform is a bounded map on a rearrangement-invariant space of functions on $\mathbb R^n$, modified by a weight, then the weight is bounded above and below and the space is equivalent to $L^2$. Also, if it is bounded from $L^p$ to $L^q$, each modified by the same weight, then the weight is bounded above and below and $1\le p=q'\le 2$. Applications prove the non-boundedness on these spaces of an operator related to the Schr\"odinger equation.
\end{abstract}
\maketitle

\section{Introduction} Plancherel's theorem and the Hausdorff-Young inequality show that if $1\le p\le2$ and $1/p + 1/q =1$, then there exists a $C$ such that for all $f\in L^1\cap L^p$,
\begin{equation}\label{HY}
\|\widehat f\|_{L^q(\mathbb R^n)}\le C\|f\|_{L^p(\mathbb R^n)}.
\end{equation}
It is well known that these are the only $(L^p, L^q)$ pairs between which the Fourier transform  is a bounded map. In addition, Theorem 1(ii) of \cite{BC} shows that if $X$ is a rearrangement invariant space of functions, then the Fourier transform is bounded on $X$ if and only if $X=L^2$, with equivalent norms.

In this paper we modify these spaces by introducing a weight function and show that the Fourier transform is bounded only if the weights are bounded above and below, reducing both problems to their respective unweighted cases. This provides a much larger class of spaces on which the Hausdorff-Young inequalities in \eqref{HY} are effectively the only possible Fourier norm inequalities. See Theorems \ref{noweight} and \ref{onlyHY}.

Following Lemma 8 of \cite{BC} we apply these results to identify a large class of spaces on which the Schr\"odinger multiplier $\exp(-4\pi^2it|y|^2)$ does not give rise to a bounded convolution operator. See Corollaries \ref{noSchr} and \ref{noHYSchr}.

If $f\colon\mathbb R^n\to\mathbb C$ is integrable, then its Fourier transform,
\[
\widehat f(x)=\int_{\mathbb R^n}e^{-2\pi ix\cdot t}f(t)\,dt,
\]
is a bounded continuous function. We restrict our attention throughout to integrable functions and avoid unnecessary extensions of the Fourier transform.

Modulation and translation operators $\varepsilon_z$ and $\tau_z$ are defined for $z\in\mathbb R^n$ by setting 
\[
\varepsilon_zf(x)=e^{-2\pi ix\cdot z}f(x)\quad\text{ and}\quad \tau_zf(x)=f(x+z), \quad\text{for }x\in\mathbb R^n.
\]
Observe that if $f$ is integrable, then $\tau_z\widehat f=(\varepsilon_z f)\,\widehat{}\,$ and $(\tau_zf)\,\widehat{}\,=\varepsilon_{-z}\widehat f$, for $z\in\mathbb R^n$.

For each $r>0$, let $Q_r=(-r/2,r/2)^n$ denote the cube in $\mathbb R^n$ with centre zero and side length $r$. Let $q_r=\chi_{Q_r}$ denote its characteristic function. Notice that if $y=(y_1,\dots,y_n)\in Q_{1/(2r)}$, then for each $j$, $-\pi/4\le\pi ry_j\le\pi/4$ so
\[
\Big|\frac{\sin(\pi r y_j)}{\pi ry_j}\Big|\ge\frac{\sin(\pi/4)}{\pi/4}\ge\frac12.
\]
Thus, for all $y,z\in\mathbb R^n$,
\begin{equation}\label{test}
|(\tau_z q_r)\,\widehat{}\,(y)|=\Big|e^{2\pi iy\cdot z}\prod_{j=1}^n\frac{\sin(\pi ry_j)}{\pi y_j}\Big|\ge \frac{r^n}{2^n}q_{1/(2r)}(y).
\end{equation}

Let $X$ be a rearrangement invariant space of complex valued functions on $\mathbb R^n$ and let $\|\cdot\|_X$ denote its norm. Directly from the definitions \cite{BS}*{I.1.1 and I.4.1}, we find that $X$ is a Banach space satisfying,
\begin{enumerate}[label=(\alph*)]
\item all characteristic functions of sets of finite measure are in $X$;
\item if $g\in X$ and $|f|\le |g|$ almost everywhere, then $f\in X$ and $\|f\|_X\le\|g\|_X$;
\item if $0\le f_k\in X$ for each $k$ and $f_k$ increases pointwise almost everywhere to $f$ as $k\to\infty$, then $f\in X$ and $\|f_k\|_X\to\|f\|_X$ as $k\to\infty$; and
\item\label{ri} if $g\in X$ and $f$ and $g$ are equimeasurable, that is, for all $\alpha>0$, the sets $\{x\in\mathbb R^n\colon|f(x)|>\alpha\}$ and $\{x\in\mathbb R^n\colon|g(x)|>\alpha\}$ have the same (Lebesgue) measure, then $f\in X$ and $\|f\|_X=\|g\|_X$.
\end{enumerate}
The associate space $X'$ of $X$ is defined to be the set of measurable $g$ such that
\[
\|g\|_{X'}=\sup\bigg\{\int_{\mathbb R^n} |fg|\colon\|f\|_X\le 1\bigg\}
\]
is finite. It is also a rearrangement invariant space and $(X')'=X$, with identical norms, see \cite{BS}*{I.2.2, I.4.2, I.2.7}. From \cite{BS}*{III.2.12 and II.5.2} we see that,
\begin{enumerate}[resume,label=(\alph*)]
\item if $T$ is a sublinear operator on $L^1+L^\infty$ that maps $L^1$ to $L^1$ and $L^\infty$ to $L^\infty$ such that $\|Tf\|_{L^1}\le \|f\|_{L^1}$ for all $f\in L^1$ and $\|Tf\|_{L^\infty}\le \|f\|_{L^\infty}$ for all $f\in L^\infty$, then $T$ maps $X$ to $X$ and $\|Tf\|_X\le \|f\|_X$ for all $f\in X$;
\item\label{fundfn} if $r>0$, then $r^n=\|q_r\|_X\|q_r\|_{X'}$; and 
\item\label{trans} if $f\in X$ and $z\in\mathbb R^n$, then $\tau_zf\in X$ and $\|\tau_zf\|_X=\|f\|_X$.
\end{enumerate}
The last statement follows readily from the translation invariance of Lebesgue measure and \ref{ri} above.

\section{Main results}

Let $\mathcal A=\{\alpha\colon\mathbb R^n\to (0,\infty):\int_{\mathbb R^n}\alpha=1\}$. These are the functions we will use to smooth weight functions by convolution. We begin with a general duality result, which we only need for rearrangement invariant spaces.

\begin{lemma}\label{duality}  Let $P$ and $Q$ be rearrangement invariant spaces of complex valued functions on $\mathbb R^n$, $U$ and $V$ be strictly positive measurable functions  and $C>0$. Suppose that if $f$ is integrable and $Uf\in P$, then $V\widehat f\in Q'$ and 
\[
\|V\widehat f\|_{Q'}\le C\|Uf\|_P.
\]  
Then, for all integrable $g$ such that $g/V\in Q$, we have $\widehat g/U\in P'$ and
\[
\|\widehat g/U\|_{P'}\le C\|g/V\|_Q.
\]
\end{lemma}
\begin{proof} Suppose $g$ is integrable and $g/V\in Q$. Choose $h\in P$ with $\|h\|_P\le1$. For each  positive integer $k$, set $h_k(x)=q_k(x)|h(x)||\widehat g(x)|/\widehat g(x)$ when $h(x)\le k U(x)$ and $\widehat g(x)\ne0$. Set $h_k(x)=0$ otherwise.
Evidently, $|h_k|\le |h|$ so $\|h_k\|_P\le1$. Also, $h_k/U$ is integrable, so
\[
\int_{\mathbb R^n}\frac{|\widehat g|}U|h_k|=\int_{\mathbb R^n}\widehat g\frac{h_k}U=\int_{\mathbb R^n}g\Big(\frac{h_k}U\Big)\!\!\widehat{\phantom{\Big )}}\,\le\|g/V\|_Q\|V(h_k/U)\,\widehat{}\,\|_{Q'}\le C\|g/V\|_Q.
\]
Letting $k\to\infty$, the monotone convergence theorem implies
\[
\int_{\mathbb R^n}\frac{|\widehat g|}U|h|\le C\|g/V\|_Q<\infty.
\]
Therefore $\widehat g/U\in P'$ and $
\|\widehat g/U\|_{P'}\le C\|g/V\|_Q$.
\end{proof}

Next we show that if a weighted Fourier inequality holds, it also holds with a smoothed weight in the codomain.

\begin{theorem}\label{smoothing} Let $X$ and $Y$ be rearrangement invariant spaces of complex valued functions on $\mathbb R^n$ and let $\alpha\in\mathcal A$. Let $u$ and $v$ be non-negative, measurable functions on $\mathbb R^n$ such that $v$ is not almost everywhere zero. Suppose there exists a $C>0$ such that, if $f$ is integrable and $uf\in X$, then $v\widehat f\in Y$ and 
\[
\|v\widehat f\|_Y\le C\|uf\|_X.
\] 
Then $u>0$ almost everywhere; $v*q_1$ is bounded above; $0<v*\alpha<\infty$ almost everywhere; and, if $f$ is integrable and $uf\in X$, then $(v*\alpha)\widehat f\in Y$ and 
\[
\|(v*\alpha)\widehat f\|_Y\le C\|uf\|_X.
\] 
\end{theorem}
\begin{proof} Suppose $f$ is integrable, $uf\in X$, $g\in Y'$ and $\|g\|_{Y'}\le1$. Interchanging the order of integration and replacing $y$ by $y+z$ produces the equation
\[
\int_{\mathbb R^n}\int_{\mathbb R^n}v(y-z)\alpha(z)\,dz|\widehat f(y)g(y)|\,dy
=\int_{\mathbb R^n}\int_{\mathbb R^n}|v(y)(\tau_z\widehat f)(y)(\tau_zg)(y)|\,dy\,\alpha(z)\,dz.
\]
Since $Y'$ is rearrangement invariant, $\|\tau_z g\|_{Y'}\le1$. Also, $\tau_z\widehat f=(\varepsilon_z f)\,\widehat{}\,$, so we have
\[
\int_{\mathbb R^n}v*\alpha(y)|\widehat f(y)g(y)|\,dy
\le \int_{\mathbb R^n}\|v\tau_z\widehat f\|_Y\alpha(z)\,dz=\int_{\mathbb R^n}\|v(\varepsilon_z f)\,\widehat{}\,\|_Y\alpha(z)\,dz.
\]
But  $|\varepsilon_z f|=|f|$, so
\[
\int_{\mathbb R^n}v*\alpha(y)|\widehat f(y)g(y)|\,dy
\le \int_{\mathbb R^n}C\|uf\|_X\alpha(z)\,dz=C\|uf\|_X<\infty.
\]
Therefore, $(v*\alpha)\widehat f\in Y$ and $\|(\alpha*v)\widehat f\|_Y\le C\|uf\|_X$.

Choose a set $F$ of finite, positive measure such that $u$ is bounded on $F$. If $u$ is zero on a set of positive measure, choose the set $F$ so that $u$ is zero on $F$. Set $f=\chi_F$. Then $f$ is bounded, integrable and not almost everywhere zero so its Fourier transform is continuous and not identically zero. Choose $a\in\mathbb R^n$ and $\delta>0$ so that if $a-y\in Q_\delta$, then $|\widehat f(y)|\ge\delta$. 
Then, for all $y,z\in\mathbb R^n$,
\[
\delta q_\delta(z-y)\le|\widehat f(a+y-z)|q_\delta(z-y)=|(\varepsilon_{a-z}f)\,\widehat{}\,(y)\,|q_\delta(z-y).
\]
The choice of $F$ ensures that $u\varepsilon_{a-z}f$ is bounded and integrable so $u\varepsilon_{a-z}f\in X$.
Therefore,
\[
\frac{\delta}{\|q_\delta\|_{Y'}} v*q_\delta(z)
\le\int_{\mathbb R^n}v(y)|(\varepsilon_{a-z}f)\,\widehat{}\,(y)|\frac{q_\delta(z-y)}{\|q_\delta\|_{Y'}}\,dy
\le\|v(\varepsilon_{a-z}f)\,\widehat{}\,\|_Y\le C\|uf\|_X.
\]
If $u$ were zero on a set of positive measure, the choice of $F$ would make the right-hand side zero and force $v$ to be almost-everywhere zero, contrary to hypothesis. Therefore $u>0$ almost everywhere. The strict positivity of $\alpha$ ensures that $v*\alpha>0$ almost everywhere.

Since $C\|uf\|_X$ finite and independent of $z$, $v*q_\delta$ is bounded above. It is a simple matter to cover $Q_1$ by finitely many translates of $Q_\delta$ to get
\[
q_1\le\sum_{j=1}^N\tau_{z_j}q_\delta,
\] 
for some finite sequence $z_1,\dots,z_N$, and hence
\[
v*q_1\le\sum_{j=1}^Nv*(\tau_{z_j}q_\delta)
=\sum_{j=1}^N\tau_{z_j}(v*q_\delta)
\]
which is bounded above. It follows that
\[
(v*\alpha)*q_1=(v*q_1)*\alpha
\]
is bounded above, which implies that $v*\alpha$ is finite almost everywhere. 
\end{proof}

Combining duality with smoothing of the codomain weight permits smoothing of both weights.

\begin{corollary}\label{cor} Under the hypotheses of the previous theorem, 
\begin{enumerate}[label=(\roman*)]
\item\label{i} If $g$ is integrable and $g/(v*\alpha)\in Y'$, then $\widehat g/u\in X'$ and
\[
\|\widehat g/u\|_{X'}\le C\|g/(v*\alpha)\|_{Y'}.
\]
\item\label{ii} $(1/u)*q_1$ is bounded above, $0<(1/u)*\alpha<\infty$ almost everywhere and, if $g$ is integrable and $g/(v*\alpha)\in Y'$, then $((1/u)*\alpha)\widehat g\in X'$ and
\[
\|((1/u)*\alpha)\widehat g\|_{X'}\le C\|g/(v*\alpha)\|_{Y'}.
\]
\item\label{iii}If $f$ is integrable and $f/((1/u)*\alpha)\in X$, then $(v*\alpha)\widehat f\in Y$ and 
\[
\|(v*\alpha)\widehat f\|_Y\le C\|f/((1/u)*\alpha)\|_X.
\]
\end{enumerate}
\end{corollary}
\begin{proof} For \ref{i}, apply Lemma \ref{duality} to the result of Theorem \ref{smoothing}. For \ref{ii}, apply Theorem \ref{smoothing} to \ref{i}. For \ref{iii}, apply Lemma \ref{duality} to \ref{ii}.
\end{proof}

Next we select a sequence of elements of $\mathcal A$ that can be used as an approximate identity on the original weights.

\begin{lemma}\label{approxid} There exist $\alpha_1,\alpha_2,\dots$ in $\mathcal A$ such that if $w\ge0$ and $w*q_1$ is bounded above, then $w*\alpha_k$ is continuous for each $k$ and as $k\to\infty$, $w*\alpha_k\to w$ almost everywhere on $\mathbb R^n$.
\end{lemma}
\begin{proof} Fix $\alpha_0\in\mathcal A$. For each positive integer $k$, set
\[
\alpha_k=\frac1kq_1*\alpha_0+\frac{k-1}kk^nq_{1/k}.
\]
We readily verify that $\alpha_k\in\mathcal A$ for each $k$. 

Since $w*q_1$ is bounded above, $w*q_1*\alpha_0$ is bounded above and $w$ is locally integrable. Lebesgue's differentiation theorem shows that for almost every $z\in\mathbb R^n$, $k^nw*q_{1/k}(z)\to w(z)$ as $k\to\infty$. It follows that, as $k\to\infty$, $w*\alpha_k\to w$ pointwise almost everywhere. 

Now fix $k$ and $z\in\mathbb R^n$. Let $B$ be an upper bound for $w*q_1$. Then for $h\in\mathbb R^n$, 
\[
|w*q_1*\alpha_0(z+h)-w*q_1*\alpha_0(z)|
\le B\int_{\mathbb R^n}|\alpha_0(z+h-y)-\alpha_0(z-y)|\,dy\to0
\]
as $h\to0$ because translation is continuous in $L^1$. So $w*q_1*\alpha_0$ is continuous at $z$.

If $h\in Q_1$ and $y\notin Q_2$, then $y+h\notin Q_1$ and we have $q_{1/k}(y+h)=q_{1/k}(y)=0$. So for sufficiently small $h$,
\[
|w*q_{1/k}(z+h)-w*q_{1/k}(z)|\le\int_{Q_2}w(z-y)|q_{1/k}(y+h)-q_{1/k}(y)|\,dy.
\]
Since $w$ is locally integrable and $|q_{1/k}(y+h)-q_{1/k}(y)|\to0$ almost everywhere as $h\to0$, the dominated convergence theorem shows $w*q_{1/k}$ is continuous at $z$. 

These combine to show that $w*\alpha_k$ is continuous on $\mathbb R^n$.
\end{proof}

Now we are ready to prove our main result: The Fourier transform is bounded on a non-trivial weighted rearrangement invariant space only if the weight is equivalent to a constant function.  
\begin{theorem}\label{noweight} Let $X$ be a rearrangement invariant space of complex valued functions on $\mathbb R^n$ and let $w$ be a non-negative measurable function on $\mathbb R^n$ that is not almost everywhere zero. Suppose there exists a $C>0$ such that, if $f$ is integrable and $wf\in X$, then $w\widehat f\in X$ and 
\[
\|w\widehat f\|_X\le C\|wf\|_X.
\] 
Then there exist positive real numbers $m$ and $M$ such that $m\le w(x)\le M$ for almost every $x\in\mathbb R^n$. Moreover, $X=L^2$ with equivalent norms.
\end{theorem}
\begin{proof} By Theorem \ref{smoothing} and Corollary \ref{cor} we get: $w>0$ almost everywhere; for each $\alpha\in\mathcal A$,  $0<w*\alpha<\infty$ and $0<(1/w)*\alpha<\infty$ almost everywhere; $w*q_1$ and $(1/w)*q_1$ are  bounded above; if $g$ is integrable and $g/(w*\alpha)\in X'$, then $((1/w)*\alpha)\widehat g\in X'$ and
\[
\|((1/w)*\alpha)\widehat g\|_{X'}\le C\|g/(w*\alpha)\|_{X'};
\]
and if $f$ is integrable and $f/((1/w)*\alpha)\in X$, then $(w*\alpha)\widehat f\in X$ and 
\[
\|(w*\alpha)\widehat f\|_X\le C\|f/((1/w)*\alpha)\|_X.
\]

Since $s+1/s\ge 2$ for $s>0$, we see that, for all $x\in\mathbb R^n$,
\[
\int_{\mathbb R^n}\int_{\mathbb R^n}\Big(\frac{w(x-z)}{w(x-y)}+\frac{w(x-y)}{w(x-z)}\Big)\alpha(y)\alpha(z)\,dy\,dz\ge\int_{\mathbb R^n}\int_{\mathbb R^n}2\alpha(y)\alpha(z)\,dy\,dz=2,
\]
which implies $1\le w*\alpha(x)(1/w)*\alpha(x)$. Therefore,
\[
\frac1{(2r)^n}\le\int_{Q_{1/(2r)}}w*\alpha(x)(1/w)*\alpha(x)\,dx
\le\|(w*\alpha)q_{1/(2r)}\|_X\|((1/w)*\alpha)q_{1/(2r)}\|_{X'}.
\]
But for any $y,z\in\mathbb R^n$, inequality \eqref{test} implies
\[
\|(w*\alpha)q_{1/(2r)}\|_X
\le\frac{2^n}{r^n}\|(w*\alpha)(\tau_zq_r)\,\widehat{}\,\|_X
\le\frac{2^nC}{r^n}\|\tau_zq_r/((1/w)*\alpha)\|_X
\]
and
\[
\|((1/w)*\alpha)q_{1/(2r)}\|_{X'}
\le\frac{2^n}{r^n}\|((1/w)*\alpha)(\tau_yq_r)\,\widehat{}\,\|_{X'}
\le\frac{2^nC}{r^n}\|\tau_yq_r/(w*\alpha)\|_{X'}.
\]
These, together with properties \ref{fundfn} and \ref{trans} above, yield
\[
1\le2^{3n}C^2\frac{\|\tau_zq_r/((1/w)*\alpha)\|_X}{\|\tau_zq_r\|_X}\frac{\|\tau_yq_r/(w*\alpha)\|_{X'}}{\|\tau_yq_r\|_{X'}}.
\]
This inequality holds for all $\alpha\in\mathcal A$ so it holds with $\alpha$ replaced by each $\alpha_k$ from the sequence given in Lemma \ref{approxid}. The lemma applies to both $w$ and $1/w$. Using the continuity of $(1/w)*\alpha_k$ and $w*\alpha_k$, and letting $r\to0$, we get
\[
1\le2^{3n}C^2\frac1{(1/w)*\alpha_k(z)}\frac1{w*\alpha_k(y)}.
\]
Now we let $k\to\infty$ to see that for almost every $z$ and almost every $y$ we have
\[
1\le2^{3n}C^2\frac{w(z)}{w(y)}.
\]
Choose $y_0$ such that the inequality holds for almost every $z$ and choose $z_0$ such that the inequality holds for almost every $y$. Then set $m=2^{-3n}C^{-2}w(y_0)$ and $M=2^{3n}C^2w(z_0)$ to get $m\le w\le M$ almost everywhere.

With this inequality in hand, the hypothesis of the theorem implies that if $f$ is integrable and $f\in X$, then $\widehat f\in X$ and $\|\widehat f\|_X\le (MC/m)\|f\|_X$. Since $X$ is rearrangement invariant, the Fourier transform extends to be bounded on all of $X$ and Theorem 1(ii) of \cite{BC} implies that $X=L^2$, with equivalent norms.
\end{proof}

Now we turn our attention to weighted Lebesgue spaces and the Hausdorff-Young inequality. It is well known that the $L^p$ spaces are rearrangement invariant, and \cite{BS}*{II.2.5} shows that $(L^p)'=L^{p'}$, with identical norms, when $1\le p\le\infty$ and $1/p+1/p'=1$. See \cite{BS}*{II.2.5}. Note that for $L^\infty$, the Banach space dual is dramatically different than the associate space, which is just $L^1$.

\begin{theorem}\label{onlyHY} Suppose $p,q\in[1,\infty]$, $w$ is a positive, measurable function on $\mathbb R^n$ and there exists a $C$ such that $\|w\widehat f\|_{L^q}\le C\|wf\|_{L^p}$
whenever $f$ is integrable and $fw\in L^p$. Then $1\le p\le 2$, $q=p'$ and there exist positive real numbers $m$ and $M$ such that $m\le w(x)\le M$ for almost every $x\in\mathbb R^n$.
\end{theorem}
\begin{proof} By Theorem \ref{smoothing} and Corollary \ref{cor} we get: $0<w$ almost everywhere; for each $\alpha\in\mathcal A$, $0<w*\alpha<\infty$ and $0<(1/w)*\alpha<\infty$ almost everywhere; $w*q_1$ and $(1/w)*q_1$ are  bounded above; if $g$ is integrable and $g/(w*\alpha)\in L^{q'}$, then $((1/w)*\alpha)\widehat g\in L^{p'}$ and
\[
\|((1/w)*\alpha)\widehat g\|_{L^{p'}}\le C\|g/(w*\alpha)\|_{L^{q'}};
\]
and if $f$ is integrable and $f/((1/w)*\alpha)\in L^p$, then $(w*\alpha)\widehat f\in L^q$ and 
\[
\|(w*\alpha)\widehat f\|_{L^q}\le C\|f/((1/w)*\alpha)\|_{L^p}.
\]

As in the proof of Theorem \ref{noweight}, we see that $1\le w*\alpha(x)(1/w)*\alpha(x)$ for all $x\in\mathbb R^n$. If both $p'$ and $q$ are finite, this implies
\[
\bigg(\frac1{(2r)^n}\bigg)^{\frac1{p'}+\frac1q}
\le\bigg(\int_{Q_{1/(2r)}}\big(w*\alpha(x)^q\big)^{p'/(p'+q)}\big((1/w)*\alpha(x)^{p'}\big)^{q/(p'+q)}\,dx\bigg)^{\frac1{p'}+\frac1q}.
\]
Applying H\"older's inequality with indices $(p'+q)/p'$ and $(p'+q)/q$ we get
\[
(2r)^{-\frac n{p'}-\frac nq}
\le\|(w*\alpha)q_{1/(2r)}\|_{L^q}\|((1/w)*\alpha)q_{1/(2r)}\|_{L^{p'}}.
\]
It is easy to verify that this inequality remains valid when one or both of $p'$ and $q$ is infinite.

But for any $y,z\in\mathbb R^n$, inequality \eqref{test} implies
\[
\|(w*\alpha)q_{1/(2r)}\|_{L^q}
\le\frac{2^n}{r^n}\|(w*\alpha)(\tau_zq_r)\,\widehat{}\,\|_{L^q}
\le\frac{2^nC}{r^n}\|\tau_zq_r/((1/w)*\alpha)\|_{L^p}
\]
and
\[
\|((1/w)*\alpha)q_{1/(2r)}\|_{L^{p'}}
\le\frac{2^n}{r^n}\|((1/w)*\alpha)(\tau_yq_r)\,\widehat{}\,\|_{L^{p'}}
\le\frac{2^nC}{r^n}\|\tau_yq_r/(w*\alpha)\|_{L^{q'}}.
\]
Since $\|\tau_zq_r\|_{L^p}=r^{n/p}$ and $\|\tau_yq_r\|_{L^{q'}}=r^{n/q'}$, the above inequalities combine to show that
\[
1\le2^{n(2+\frac1{p'}+\frac1q)}C^2\frac{\|\tau_zq_r/((1/w)*\alpha)\|_{L^p}}{\|\tau_zq_r\|_{L^p}}\frac{\|\tau_yq_r/(w*\alpha)\|_{L^{q'}}}{\|\tau_yq_r\|_{L^{q'}}}.
\]
As in the proof of Theorem \ref{noweight}, we replace $\alpha$ by $\alpha_k$, let $r\to0$ to get
\[
1\le2^{n(2+\frac1{p'}+\frac1q)}C^2\frac1{(1/w)*\alpha_k(z)}\frac1{w*\alpha_k(y)},
\]
and let $k\to\infty$ to get
\[
1\le2^{n(2+\frac1{p'}+\frac1q)}C^2\frac{w(z)}{w(y)}.
\]
It follows as above that there exist positive real numbers $m$ and $M$ such that $m\le w(x)\le M$ for almost every $x\in\mathbb R^n$. The hypothesis of the theorem now implies that if $f$ is integrable and $f\in L^p$, then $\widehat f\in L^q$ and
\[
\|\widehat f\|_{L^q}\le (MC/m)\|\widehat f\|_{L^p}.
\]
This can only happen when $1\le p=q'\le 2$.

To see this well known fact we may use \eqref{test} with $z=0$, taking $f=q_r$ to see that $q=p'$ is a necessary condition for \eqref{HY}. Also, Theorem 1(i) of \cite{BC} shows that $L^1+L^2$ is the largest rearrangement invariant space which the Fourier Transform maps into a space of locally integrable functions. If $p>2$ then $L^p$ is not a subset of $L^1+L^2$, making $p\le 2$ also a necessary condition for \eqref{HY}.
\end{proof}

\section{The Schr\"odinger Multiplier}

The Fourier transform separates variables in the Schr\"odinger equation
\begin{equation}\label{Schr}
\partial_tu(t,x)=i\Delta_xu(t,x).
\end{equation}
The resulting multiplier is $\exp(-4\pi^2it|y|^2)$. 
That is, if $u(0,x)=h(x)$, with $h$ integrable, the solution to \eqref{Schr} is $u(t,x)=S_th(x)$, where the operator $S_t$ is defined by
\[
\widehat{S_th}(y)=\exp(-4\pi^2it|y|^2)\widehat h(y).
\]
The operator $S_t$ can be extended (along with the Fourier transform) to spaces other than $L^1$ in order to solve the  Schr\"odinger equation for non-integrable initial data. Concrete extensions of $S_t$ to other spaces of functions follow from boundedness of the operator on the integrable functions in the space. So it is natural to investigate the spaces on which $S_t$ is bounded. Lemma 8 of \cite{BC} showed that $S_t$ is not bounded on any rearrangement invariant space of functions unless the space is $L^2$. Here we modify  rearrangement invariant spaces with a weight function and show that a similar negative result holds on this larger class. 

\begin{corollary}\label{noSchr}
Let $X$ be a rearrangement invariant space of complex valued functions on $\mathbb R^n$ and let $t>0$. Fix a non-negative measurable function $w$ and set $w_t(x)=w(x/(4\pi t))$ for $x\in\mathbb R^n$. Suppose there exists a $C>0$ such that if $h$ is integrable and $wh\in X$, then $w_tS_th\in X$ and 
\[
\|w_tS_th\|_X\le C\|wh\|_X.
\]
Then $w$ is bounded above and below, and $X=L^2$ with equivalent norms. 
\end{corollary}
\begin{proof} Since $X$ is rearrangement invariant, the dilation map $f(x)\to f(x/(4\pi t))$ is bounded on $X$. (It is trivially bounded on $L^\infty$ and a simple change of variable shows that it is bounded on $L^1$.) Let $M$ be a bound for this map. 

A Fourier transform calculation shows that the operator $S_t$ can be written as the convolution 
\[
S_th(y)=(4\pi it)^{-n/2}\int_{\mathbb R^n}\exp\bigg(i\frac{|y-x|^2}{4t}\bigg)h(x)\,dx
\]
and the simplification in Lemma 8 of \cite{BC}, taking $\sigma(x)=\exp(i|x|^2/(4t))$, gives
\[
S_th(y)=(4\pi it)^{-n/2}\sigma(y)\widehat{\sigma h}(y/(4\pi t)).
\]

Suppose $f$ is integrable and $wf\in X$. Set $h=f/\sigma$. Since $|\sigma(x)|=1$ for all $x$, $h$ is integrable and $wh\in X$. Thus $\|w_tS_th\|_X\le C\|wh\|_X$. The calculation above shows that
\[
w(y/(4\pi t))|\widehat f(y/(4\pi t))|=(4\pi t)^{n/2}w_t(y)|S_th(y)|.
\]
Since $w_tS_th\in X$ and $X$ is closed under dilations, $w\widehat f\in X$. Moreover,
\[
\|w\widehat f\|_X\le M(4\pi t)^{n/2}\|w_tS_th\|_X\le M(4\pi t)^{n/2}C\|wh\|_X=M(4\pi t)^{n/2}C\|wf\|_X.
\]
Now we apply Theorem \ref{noweight} to see that $w$ is bounded above and below, and $X=L^2$ with equivalent norms.
\end{proof}

We state the next corollary without proof, as it follows from Theorem \ref{onlyHY} in the same way as Corollary \ref{noSchr} follows from Theorem \ref{noweight}.
\begin{corollary}\label{noHYSchr} Let $p,q\in[1,\infty]$ and let $t>0$. Fix a non-negative weight function $w$ and set $w_t(x)=w(x/(4\pi t))$ for all $x\in\mathbb R^n$. Suppose there exists a $C$ such that if $h$ is integrable and $wh\in L^p$, then $w_tS_th\in L^q$ and 
\[
\|w_tS_th\|_{L^q}\le C\|wh\|_{L^p}.
\]
Then $w$ is bounded above and below, $1\le p\le2$ and $q=p'$. 
\end{corollary}

\end{document}